\documentclass{amsart}

\usepackage{latexsym,amssymb}
\usepackage{amsmath,amsthm}

\usepackage{amsfonts}
\usepackage{graphicx}   
\usepackage{url}

\def\glim{\mathop{\text{\normalfont $\Gamma-$lim}}}

\def\diam{\mathop{\text{\normalfont diam}}}
\def\pint{\operatorname {--\!\!\!\!\!\int\!\!\!\!\!--}}

\newcommand{\R}{\mathbb{R}}

\newcommand{\N}{\mathbb{N}}

\newcommand{\ve}{\varepsilon}
\newcommand{\ito}{\infty}

\newtheorem{thm}{Theorem}[section]

\newtheorem{lemp}[thm]{Lemma}
\newtheorem{prop}[thm]{Proposition}
	
\theoremstyle{definition}
\newtheorem{de}[thm]{Definition}
\theoremstyle{remark}
\newtheorem{rk}[thm]{Remark}

\numberwithin{equation}{section}

\begin{document}

\title{Some nonlocal optimal design problems}

\author[J. Fern\'andez Bonder and J.F. Spedaletti]{Juli\'an Fern\'andez Bonder and Juan F. Spedaletti}

\address[J. Fern\'andez Bonder]{Departamento de Matem\'atica FCEN - Universidad de Buenos Aires and IMAS - CONICET. Ciudad Universitaria, Pabell\'on I (C1428EGA)
Av. Cantilo 2160. Buenos Aires, Argentina.}

\email{jfbonder@dm.uba.ar}

\urladdr{http://mate.dm.uba.ar/~jfbonder}

\address[J. F. Spedaletti]{Departamento de Matem\'atica, Universidad Nacional de San Luis and IMASL - CONICET. Ej\'ercito de los Andes 950 (D5700HHW), San Luis, Argentina.}

\email{jfspedaletti@unsl.edu.ar}

%35J92  	Quasilinear elliptic equations with $p$-Laplacian
%49R05  	Variational methods for eigenvalues of operators
%47A75  	Eigenvalue problems
%47J10  	Nonlinear spectral theory, nonlinear eigenvalue problems
%35P30  	Nonlinear eigenvalue problems, nonlinear spectral theory


\begin{abstract}
In this paper we study two optimal design problems associated to fractional Sobolev spaces $W^{s,p}(\Omega)$. Then we find a relationship between these two problems and finally we investigate the convergence when $s\uparrow 1$.
\end{abstract}

\maketitle

%%%%%%%%%%%%%%%%%%%%%%%%%
%%
%% INTRODUCTION
%%
%%%%%%%%%%%%%%%%%%%%%%%%%
\section{Introduction}
Let $\Omega \subset \R^n$ be an open, connected and bounded set. For $0<s<1$ and $1<p<\ito$ we consider the fractional Sobolev space $W^{s,p}(\Omega)$ defined as follows
\begin{equation}\label{Sobolevfraccionarioclasico}
W^{s,p}(\Omega)=\left \{u\in L^{p}(\Omega)\colon \frac{|u(x)-u(y)|}{|x-y|^{\frac{n}{p}+s}}\in L^{p}(\Omega\times \Omega)\right \},
\end{equation}
endowed with the natural norm
\begin{equation}\label{normaSobolevclasico}
\|u\|_{W^{s,p}(\Omega)}=\left(\int_\Omega |u|^p\,dx+\iint_{\Omega\times\Omega} \frac{|u(x)-u(y)|^p}{|x-y|^{n+sp}}\,dx \,dy \right)^{1/p}.
\end{equation}
The term
\begin{equation}\label{seminormaSobolev}
[u]_{W^{s,p}(\Omega)}^p = [u]_{s,p}^p= \iint_{\Omega\times\Omega} \frac{|u(x)-u(y)|^p}{|x-y|^{n+sp}}\,dx \,dy,
\end{equation}
is called the {\em Gagliardo seminorm} of $u$. We refer the interested reader to \cite{DiNezza-Palatucci-Valdinoci} for a throughout introduction to these spaces.

The purpose of this paper is to analyze some optimization problems related to the best Poincar\'e constant in these spaces. First, we consider the following problem: given a measurable set $A\subset \Omega$, we define the optimal Poincar\'e constant $\lambda_s(A)$ as the number
\begin{equation}\label{lambdaA}
\lambda_s(A):= \inf\left\{\frac{\frac12[v]_{s,p}^p}{\|v\|_p^p}\colon v\in W^{s,p}(\Omega),\ v=0 \text{ a.e. in } A\right\}.
\end{equation}
This constant is the largest possible one in Poincar\'e's inequality
$$
\lambda \int_\Omega |v|^p\, dx \le \frac12\iint_{\Omega\times\Omega} \frac{|v(x)-v(y)|^p}{|x-y|^{n+sp}}\,dx \,dy 
$$
for every function $v\in W^{s,p}(\Omega)$ that vanishes on the set $A$.

Also, this constant can be seen as the first eigenvalue of a fractional $p-$laplace type equation. See next section.

The first problem that we want to address is to minimize this constant $\lambda_s(A)$ with respect to the set $A$ in the class of measurable sets of fixed measure. That is, we take $\alpha\in (0,1)$ and define the class
$$
\mathcal{A}_\alpha := \{A\subset \Omega\colon A \text{ measurable and } |A|=\alpha |\Omega|\}.
$$
So our optimization problem reads, find an {\em optimal set} $A_s\in \mathcal{A}_\alpha$ such that
\begin{equation}\label{hard.opt}
\lambda_s(A_s) = \Lambda_s(\alpha) := \inf\{\lambda_s(A)\colon A\in \mathcal{A}_\alpha\}.
\end{equation}
This problem is called the {\em Hard Obstacle Problem} since the optimal set $A$ can be seen as the obstacle where the solution is forced to vanish. 

In the case $s=1$, that is when the classical Sobolev spaces are consider, some related problems were studied in \cite{Bonder-Rossi-Wolanski2, Bonder-Rossi-Wolanski1}. In those papers it was shown that there exists an optimal configuration, and some properties of optimal configurations and of their associated extremals were obtained. We refer to the interested reader to the above mentioned papers.

For this hard obstacle problem, our main result reads:
\begin{thm}\label{teo.hard.intro}
Let $\alpha\in (0,1)$, $0<s<1<p<\ito$ and $\Omega\subset\R^n$ be a bounded open set. Then, there exists a measurable set $A_s\subset\Omega$ such that $$
|A_s|=\alpha |\Omega|\quad \text{and}\quad \Lambda_s(\alpha)=\lambda_s(A_s)
$$ 
where $\lambda_s(A)$ and $\Lambda_s(\alpha)$ are given by \eqref{lambdaA} and \eqref{hard.opt} respectively. 

Moreover, if $u_s\in W^{s,p}(\Omega)$ is an extremal associated to $\lambda_s(A)$, then 
$$
A_s = \{u_s=0\}\cap\Omega \text{ a.e.}
$$
\end{thm}

Related to this optimization problem, is the following variant that sometimes is referred to as the {\em Soft Obstacle Problem}. That is, given $\sigma>0$ and $A\subset \Omega$ measurable, we look for the best optimal constant in the following Poincar\'e-type inequality
$$
\lambda \int_\Omega |v|^p\, dx \le \frac12\iint_{\Omega\times\Omega} \frac{|v(x)-v(y)|^p}{|x-y|^{N+sp}}\, dxdy + \sigma \int_\Omega |v|^p\, \chi_A dx,
$$
where $\chi_A$ denotes the characteristic function of the set $A\subset \Omega$.

That is
$$
\lambda_s(\sigma, A) := \inf\left\{\frac{\frac12[v]_{s,p}^p + \sigma \|v\|_{p, A}^p}{\|v\|_p^p}\colon v\in W^{s,p}(\Omega)\right\},
$$
where $\|v\|_{p;A}^p = \int_A |v|^p\, dx$.

The soft obstacle problem then consists on minimizing the constants $\lambda_s(\sigma, A)$ among sets $A\in \mathcal{A}_\alpha$.

The hard and soft obstacle problems are related by the fact that the term $\sigma\|v\|_{p,A}$ can be seen as a penalization term and then (heuristically),
$$
\lambda_s(\sigma, \chi_A) \to \lambda_s(A) \quad \text{as}\ \sigma\to\ito.
$$
So the next point of the paper is to make this fact rigorous.

Since the set of characteristic functions is not closed under any reasonable topology, it is necesary to relax the problem and work within the class of functions $\phi\in L^\ito(\Omega)$ such that $0\le \phi\le 1$ which is the weak-* closure of the sets of characteristic functions.

So, given any such $\phi\in L^\ito(\Omega)$ and a constant $\sigma>0$, the problem that we address is the following: define the best constant in a Poincar\'e-type inequality as
\begin{equation}\label{lambdaphi}
\lambda_s(\sigma, \phi) := \inf\left\{\frac{\frac12[v]_{s,p}^p + \sigma \|v\|_{p,\phi}^p}{\|v\|_p^p}\colon v\in W^{s,p}(\Omega)\right\},
\end{equation}
where
$$
 \|v\|_{p,\phi}^p = \int_\Omega |v|^p\, \phi\, dx,
$$
denote the class of admissible {\em potentials} $\phi$ by
$$
\mathcal{B}_\alpha := Š\{\phi\in L^\ito(\Omega)\colon 0\le\phi\le 1,\ \|\phi\|_1=\alpha |\Omega|\},
$$
and consider the optimization problem
\begin{equation}\label{soft.opt}
\Lambda_s(\sigma,\alpha) := \inf\{\lambda_s(\sigma, \phi)\colon \phi\in \mathcal{B}_\alpha\}.
\end{equation}
Then look for an optimal potential $\phi_\sigma$ such that
$$
\lambda_s(\sigma, \phi_\sigma) = \Lambda_s(\sigma,\alpha).
$$
Recall that $\chi_A\in \mathcal{B}_\alpha$ if and only if $A\in \mathcal{A}_\alpha$.

For this problem we show the existence of this optimal potential and, moreover, we prove that $\phi_\sigma = \chi_{A_\sigma}$ for some $A_\sigma\in \mathcal{A}_\alpha$, so we recover a solution to our original soft obstacle problem. Finally, we show that
$$
\Lambda_s(\sigma, \alpha) \to \Lambda_s(\alpha) \text{ and } \chi_{A_\sigma} \to \chi_{A_s} \text{ in } L^1(\Omega) \quad \text{as } \sigma\to\ito,
$$
where $A_s\in \mathcal{A}_\alpha$ is an optimal configuration for $\Lambda_s(\alpha)$. See \cite{Bonder-Giubergia-Mazzone} for related results in a local problem.

Our results for the soft obstacle problem are summarized in the next theorem.
\begin{thm}\label{teo.soft.intro}
Let $0<s<1<p<\ito$ and $\Omega\subset \R^n$ be a bounded open set with Lipschitz boundary. Let also $\sigma>0$ and $0<\alpha<1$ be fixed. Then there exists $A_{s,\sigma}\in \mathcal{A}_\alpha$ such that
$$
\lambda_s(\sigma, \chi_{A_{s,\sigma}}) = \Lambda_s(\sigma,\alpha),
$$
where $\lambda_s(\sigma,\phi)$ and $\Lambda_s(\sigma,\alpha)$ are given by \eqref{lambdaphi} and \eqref{soft.opt} respectively.

Moreover, if $u_{\sigma}\in W^{s,p}(\Omega)$ is an extremal associated to $\lambda_s(\sigma, \chi_{A_{s,\sigma}})$, then there exists $t>0$ such that
$$
\{u_\sigma<t\}\subset A_{s,\sigma}\subset \{u_\sigma \le t\}.
$$

For the limit $\sigma\to\ito$, there holds that
$$
\lim_{\sigma\to\ito} \Lambda_s(\sigma, \alpha) = \Lambda_s(\alpha),
$$
and any family of optimal pairs $\{(u_\sigma, \chi_{A_{s,\sigma}})\}_{\sigma>0} \subset W^{s,p}(\Omega)\times \mathcal{B}_\alpha$ is precompact in the strong topology of $W^{s,p}(\Omega)$ in the first variable and the weak* topology of $L^\ito(\Omega)$ in the second variable.

Finally, any acumulation point of the family of optimal pairs has the form $(u,\chi_A)$ where $A\in \mathcal{A}_\alpha$, $A=\{u=0\}\cap\Omega$ and $u$ is an extremal for $\Lambda_s(\alpha)$.
\end{thm}

The soft obstacle problem in the relaxed form of \eqref{lambdaphi} also appears naturally in the study of the (time independent) fractional Schroedinger equation (for $p=2$) where it is of special interest to analyze the behavior of the eigenvalues (so-called fundamental states of the system). See, for instance \cite{Laskin}.

In this context, is relevant also to consider potentials $\phi$ that are allowed to change sign. Moreover, the bound $\phi\le 1$ is meaningless there.

An optimization problem related to the fractional Schroedinger eigenvalues with $L^q-$bounds ($q>1$) on the potential functions has been recently carried out in \cite{DP-FB-R}.

With the same methods presented in this paper, the minimization problem for the first eigenvalue of the fractional Schroedinger equation can be fully analyzed in the case of nonnegative, uniformly bounded potentials with prescribed $L^1-$norm. We leave the details to the interested reader.

To finish the paper, we analyze the connection between the hard obstacle problem \eqref{hard.opt} and its classical counterpart, when $s=1$. Therefore, we analyze the asymptotic behavior as $s\uparrow 1$ for \eqref{hard.opt} and based on some Gamma-convergence results due to A. Ponce in \cite{Ponce} we are able to prove the convergence of the nonlocal model to the local one.

To be more precise, let us define the constant
$$
\Lambda(\alpha) = \inf\left\{\tfrac12 \|\nabla v\|_p^p\colon v\in W^{1,p}(\Omega)\cap E_\alpha\right\}, 
$$
where 
$$
E_\alpha = \{v\in L^p(\Omega)\colon |\{v=0\}\cap\Omega|\ge \alpha |\Omega|,\ \|v\|_p=1\}.
$$
So we obtain the following behavior as $s\uparrow 1$.
\begin{thm}\label{teo.s1.intro}
Let $1<p<\ito$, $\Omega\subset \R^n$ be open and bounded with Lipschitz boundary and let $K(n,p)$ be the constant defined as
$$
K(n,p)=\pint_{S^{n-1}}|z_n|^p\, dS_z.
$$
Then
$$
\lim_{s\uparrow 1} (1-s)\Lambda_s(\alpha) = K(n,p)\Lambda(\alpha).
$$
Moreover, if $u_s\in W^{s,p}(\Omega)\cap E_\alpha$ is an extremal for $\Lambda_s(\alpha)$, then $\{u_s\}_{0<s<1}$ is precompact in $L^p(\Omega)$ and every accumulation point is an extremal for $\Lambda(\alpha)$. 

Finally, if $A_s\subset \Omega$ is an optimal set for $\Lambda_s(\alpha)$ then, up to some subsequence, there exists $A\subset \Omega$ such that
$$
\chi_{A_s}\to \chi_A \quad \text{strongly in } L^1(\Omega),
$$ 
and $A$ is optimal for $\Lambda(\alpha)$.
\end{thm}

\subsection*{Organization of the paper} After this introduction, the rest of the paper is organized as follows.

We begin in Section 2 with a rather large section where all the preliminaries on fractional Sobolev spaces and on the fractional $p-$laplacian are collected. This section contains almost no new material and an expert on the field can safely skip it an move directly to the next sections. We choose to include it because some of the results (specially subsection 2.1) are scattered in the literature and we weren't able to find a precise reference for those.

Section 3 contains the main results of the paper. Namely the study of the Hard and Soft Obstacle Problems \eqref{hard.opt} and \eqref{soft.opt} respectively and the connection between them.

Finally, in Section 4, we analyze the asymptotic behavior of the Hard Obstacle Problem when $s\uparrow 1$.

%%%%%%%%%%%%%%%%%%%%%%%%%
%%
%% PRELIMINARIES
%%
%%%%%%%%%%%%%%%%%%%%%%%%%
\section{Preliminaries}
In this section, we review some definitions on fractional Sobolev spaces and on the $p-$fractional Laplace operator. We believe that most of this results are known to experts and constitute part of the ``folklore'' on the subject but since we were not able to find a precise reference for these, we have chosen to include proofs of most of the facts that are needed.

%%%%%%%%%%%%%%%%%%%%
%% REGIONAL (p,s)-LAPLACIAN
%%%%%%%%%%%%%%%%%%%%
\subsection{The regional $(s,p)-$laplacian}
We begin with the definition of the fractional $p-$laplacian that we use in this paper. This operator is some times denoted as the {\em regional} fractional $p-$laplacian.

For any smooth and bounded function $u$ ($C^2(\Omega)\cap L^\infty(\Omega)$ is enough), we define the regional $(s,p)-$laplacian as
\begin{equation}\label{eq.plap}
\begin{split}
(-\Delta_{p,\Omega})^s u (x) &:= \text{p.v.} \int_\Omega \frac{|u(x)-u(y)|^{p-2} (u(x)-u(y))}{|x-y|^{n+sp}}\, dy\\
&= \lim_{\ve\downarrow 0} \int_{\Omega\setminus B_\ve(x)}\frac{|u(x)-u(y)|^{p-2} (u(x)-u(y))}{|x-y|^{n+sp}}\, dy,
\end{split}
\end{equation}
for any $x\in \Omega$.

Let us first see that, in the case $2\le p<\infty$, this operator is well defined for smooth enough functions.
\begin{lemp}\label{plap.bien}
Let $0<s<1$ and $2\le p<\infty$ be fixed and let $\Omega\subset\R^n$ be open. Then the operator $(-\Delta_{p,\Omega})^su(x)$ is well defined for $u\in C^2(\Omega)\cap L^\infty(\Omega)$.
\end{lemp}

\begin{proof}
Let $\ve_0>0$ be such that $B_{\ve_0}(x)\subset\subset\Omega$. Now, as $u\in L^\infty(\Omega)$ we have that
$$
\int_{\Omega\setminus B_{\ve_0}(x)} \frac{|u(x)-u(y)|^{p-1}}{|x-y|^{n+sp}}\, dy <\infty.
$$

Now, for $0<\ve<\ve_0$, we have
\begin{align*}
\int_{\ve<|x-y|<\ve_0} &\frac{|u(x)-u(y)|^{p-2} (u(x)-u(y))}{|x-y|^{n+sp}}\, dy\\
&= \int_{\ve<|z|<\ve_0} \frac{|u(x)-u(x+z)|^{p-2} (u(x)-u(x+z))}{|z|^{n+sp}}\, dz\\
&= \int_{\ve<|z|<\ve_0} \frac{|u(x)-u(x-z)|^{p-2} (u(x)-u(x-z))}{|z|^{n+sp}}\, dz.
\end{align*}
To simplify the notation, let us denote $\phi_p(t) = |t|^{p-2}t$. Therefore this last quantity equals
\begin{equation}\label{eq.intermedia}
\frac12 \int_{\ve<|z|<\ve_0} \frac{\phi_p(u(x)-u(x+z)) - \phi_p(u(x)-u(x-z))}{|z|^{n+sp}}\, dz
\end{equation}

Now, we also define
$$
\varphi(t) = \phi_p((u(x) - u(x-z) + t(u(x+z) - 2u(x) + u(x-z))).
$$
So \eqref{eq.intermedia} can be written as
$$
\frac12\int_{\ve<|z|<\ve_0} \frac{\varphi(1) - \varphi(0)}{|z|^{n+sp}}\, dz = \frac12\int_{\ve<|z|<\ve_0}\int_0^1 \varphi'(t)\, dt\, |z|^{-(n+sp)}\, dz
$$
and performing the computations, this equals
\begin{equation}\label{eq.segunda}
\frac12\int_{\ve<|z|<\ve_0}\int_0^1 \frac{|(t-1) D_{-z} u(x) + t D_z u(x)|^{p-2} D^2_zu(x)}{|z|^{n+sp}}\, dt\, dz,
\end{equation}
where $D_z u(x) = u(x+z)-u(x)$ and $D^2_z u(x) = u(x+z) - 2u(x) + u(x-z) = D_z u(x) + D_{-z} u(x)$.

From this last expression, since $2\le p<\infty$, we observe that 
$$
|(t-1) D_{-z} u(x) + t D_z u(x)|^{p-2}\le C |z|^{p-2},
$$ 
where $C$ depends on the Lipschitz constant of $u$ and on $p$, and 
$$
|D^2_z u(x)|\le C' |z|^2,
$$
where $C'$ depends on the $C^2-$norm of $u$.

Putting all of these together we find out that there exists a constant $C$ depending only on the $C^2-$norm of $u$ and on $p$ such that \eqref{eq.segunda} is bounded by
$$
C\int_{\ve<|z|<\ve_0} |z|^{-n+p(1-s)}\, dz
$$
and since this last term converges as $\ve\downarrow 0$, the lemma follows.
\end{proof}

Lemma \ref{plap.bien} tells us that the regional $(s,p)-$laplacian is well defined for regular functions. Unfortunately, this is not enough. We need to know how acts on measurable functions. Moreover, we need to extend the range of the exponent $p$ to the whole interval $(1,\infty)$. We perform this task in the next lemma.

\begin{lemp}\label{plap.distribucional}
Let $0<s<1<p<\infty$ be fixed and let $\Omega\subset\R^n$ be open. For every $u\in W^{s,p}(\Omega)$, the regional $(s,p)-$laplacian given by \eqref{eq.plap} defines a distribution $\mathcal D'(\Omega)$. Moreover,
$$
\langle (-\Delta_{p,\Omega})^s u, \phi\rangle = \frac12 \iint_{\Omega\times \Omega} \frac{|u(x)-u(y)|^{p-2} (u(x)-u(y))(\phi(x)-\phi(y))}{|x-y|^{n+sp}}\, dx\, dy,
$$
for every $\phi\in C^\infty_c(\Omega)$.
\end{lemp}

\begin{proof}
Given $u\in W^{s,p}(\Omega)$, for any $\ve>0$ we define $T_\ve u$ as
$$
T_\ve u (x) = \int_{\Omega\setminus B_\ve(x)} \frac{|u(x)-u(y)|^{p-2}(u(x)-u(y))}{|x-y|^{n+sp}}\, dy.
$$
We claim that $T_\ve u\in L^{p'}(\Omega)$. In fact, by H\"older's inequality,
\begin{align*}
|T_\ve u(x)| &\le \int_{\Omega\setminus B_\ve(x)} \frac{|u(x)-u(y)|^{p-1}}{|x-y|^{n+sp}}\, dy\\
&\le \left(\int_{\Omega\setminus B_\ve(x)} \frac{|u(x)-u(y)|^p}{|x-y|^{n+sp}}\, dy\right)^{\frac{1}{p'}} \left(\int_{|x-y|\ge \ve} \frac{1}{|x-y|^{n+sp}}\, dy\right)^{\frac{1}{p}}.
\end{align*}
Let us denote
$$
C = C(n,s,p) = \frac{n\omega_n}{sp},
$$
then it is easy to see that
$$
\int_{|x-y|\ge \ve} \frac{1}{|x-y|^{n+sp}}\, dy = C \ve^{-sp}.
$$
So, we easily conclude that
$$
\|T_\ve u\|_{p', \Omega} \le C^{\frac{1}{p}} \ve^{-s} [u]_{s,p}^{\frac{p}{p'}}.
$$

Therefore, $T_\ve u$ induces a distribution as
$$
\langle T_\ve u, \phi\rangle = \int_{\Omega} T_\ve u \phi\, dx = \int_\Omega\int_{\Omega\setminus B_\ve(x)} \frac{|u(x)-u(y)|^{p-2}(u(x)-u(y))}{|x-y|^{n+sp}} \phi(x)\, dy dx.
$$
It is easy to check, by using Fubini's Theorem, that one also has
$$
\langle T_\ve u, \phi\rangle = -\int_\Omega\int_{\Omega\setminus B_\ve(x)} \frac{|u(x)-u(y)|^{p-2}(u(x)-u(y))}{|x-y|^{n+sp}} \phi(y)\, dy dx,
$$
and so
$$
\langle T_\ve u, \phi\rangle = \frac12 \int_\Omega\int_{\Omega\setminus B_\ve(x)}\frac{|u(x)-u(y)|^{p-2}(u(x)-u(y))(\phi(x)-\phi(y))}{|x-y|^{n+sp}}\, dy dx.
$$
Observe that, by H\"older's inequality, the integrand is in $L^1(\Omega\times \Omega)$. In fact,
\begin{align*}
\iint_{\Omega\times\Omega} \frac{|u(x)-u(y)|^{p-1} |v(x)-v(y)|}{|x-y|^{n+sp}}\, dxdy &= \iint_{\Omega\times\Omega} \frac{|u(x)-u(y)|^{p-1}}{|x-y|^{\frac{n+sp}{p'}}} \frac{|v(x)-v(y)|}{|x-y|^\frac{n+sp}{p}}\, dxdy\\
&\le [u]_{s,p}^{p-1} [v]_{s,p}.
\end{align*}

Therefore, using the Dominated Convergence Theorem, one concludes that
\begin{align*}
\langle (-\Delta_{p,\Omega})^s u, \phi\rangle &= \lim_{\ve\downarrow 0} \langle T_\ve u, \phi\rangle\\
&= \frac12 \iint_{\Omega\times \Omega}\frac{|u(x)-u(y)|^{p-2}(u(x)-u(y))(\phi(x)-\phi(y))}{|x-y|^{n+sp}}\, dx dy.
\end{align*}
This finishes the proof.
\end{proof}

\begin{rk}\label{def.lap}
From the proof of Lemma \ref{plap.distribucional} one observe that the regional $(p,s)-$laplacian is a bounded operator between $W^{s,p}(\Omega)$ and its dual $[W^{s,p}(\Omega)]'$.
\end{rk}

With all of these preliminaries, we establish the definition of {\em weak solution} for the regional $(p,s)-$laplacian.
\begin{de}\label{de.sol}
Let $0<s<1<p<\infty$ be fixed and let $\Omega\subset\R^n$ be open. Given $f\in L^{p'}(\Omega)$ (or more generally, $f\in [W^{s,p}(\Omega)]'$), we say that $u\in W^{s,p}(\Omega)$ is a weak solution of 
\begin{equation}\label{eq.sp}
(-\Delta_{p,\Omega})^s u = f \quad \text{in }\Omega,
\end{equation}
if the equality holds in the distributional sense. That is, if
\begin{equation}\label{eq.sp.weak}
\frac12\iint_{\Omega\times\Omega} \frac{|u(x)-u(y)|^{p-2} (u(x)-u(y))(v(x)-v(y))}{|x-y|^{n+sp}}\, dx\, dy = \int_\Omega f v\, dx,
\end{equation}
for every $v\in W^{s,p}(\Omega)$.
\end{de}

\begin{rk}
This problem is analog to the non-homogeneous Neumann problem in the classical local setting.
\end{rk}

In the study of the hard obstacle problem, we need to look for solutions of a mixed boundary value problem. We need a definition for solutions of such problems. 

\begin{de}\label{de.mixto}
Let $0<s<1<p<\infty$ be fixed and let $\Omega\subset\R^n$ be open. Let $A\subset \Omega$ be measurable and let $f\in L^{p'}(\Omega)$ (or more generally, $f\in [W^{s,p}(\Omega)]'$). We define
$$
W^{s,p}_A(\Omega) := \{v\in W^{s,p}(\Omega)\colon v=0 \text{ a.e. in } A\}.
$$
Then, we say that $u\in W^{s,p}_A(\Omega)$ is a weak solution of the mixed boundary value problem
\begin{equation}\label{mixed}
\begin{cases}
(-\Delta_{p,\Omega})^s u = f & \text{in }\Omega\setminus A\\
u = 0 & \text{in }A,
\end{cases}
\end{equation}
if \eqref{eq.sp.weak} holds for every $v\in W^{s,p}_A(\Omega)$.
\end{de}

In what follows we need a version of the strong minimum principle for solutions of \eqref{mixed}. Following ideas in \cite{DiCastro-Kuusi-Palatucci} we can provide the following lemma needed in order to prove our version of the strong minimum principle. 

\begin{lemp}[Logarithmic lemma]\label{logaritmic}
Let $0<s<1<p<\ito$ be fixed and let $\Omega\subset \R^n$ be an open set. Let $A\subset \Omega$ be closed. Assume that $f\in L^{p'}(\Omega)$ is nonnegative and that $u\in W_A^{s,p}(\Omega)$ is a nonnegative weak solution of \eqref{mixed} in the sense of Definition \ref{de.mixto}. If $B_R(x_0)\subset\subset \Omega\setminus A$, then the following estimate holds: for any $B_r(x_0)\subset B_{R/2}(x_0)$ and every $\delta>0$, 
\begin{equation}\label{deslogaritmica}
\iint_{B_r(x_0)\times B_r(x_0)}\left |\log\left (\frac{u(x)+\delta}{u(y)+\delta}\right )\right |^p\frac{1}{|x-y|^{n+sp}}\,dx\,dy\le C r^{n-sp},
\end{equation}
where $C=C(n,p,s)>0$.
\end{lemp}

\begin{proof}
Let $x_0\in \Omega\setminus A$ and $R>0$ be such that $B_R(x_0)\subset\subset \Omega\setminus A$. In the rest of the proof, to simplify notation, every ball will be centered at $x_0$.

Let $\delta >0$ be a real parameter and let $\phi \in C_0^\ito(B_{3r/2})$ for $r>0$ be such that  
$$
0\leq \phi\leq 1, \phi \equiv 1\text{ in }B_r \text{ and } |\nabla\phi |< cr^{-1} \text{ in } B_{3r/2}\subset B_{R/2}.
$$
Now we use the weak formulation \eqref{eq.sp.weak} with the function 
$$
\eta=(u+\delta)^{1-p}\phi ^p.
$$
Observe that the test function $\eta$ is well defined since the function $u\geq 0$ in the support of the function $\phi$, and that $\eta=0$ on $A$.

So, we get
\begin{equation}\label{usodebil}
\begin{split}
0\leq \int_\Omega f \eta\, dx =& \frac12\iint_{\Omega\times \Omega}\frac{|u(x)-u(y)|^{p-2}(u(x)-u(y))(\eta(x)-\eta(y))}{|x-y|^{n+sp}}\, dxdy\\
 =& I_1 + I_2 + I_3, 
\end{split}
\end{equation}
where
\begin{align*}
I_1 &= \frac12\iint_{B_{2r}\times B_{2r}}\frac{|u(x)-u(y)|^{p-2}(u(x)-u(y))}{|x-y|^{n+sp}}\left [ \frac{\phi(x)^p}{(u(x)+\delta)^{p-1}}-\frac{\phi(y)^p}{(u(y)-\delta)^{p-1}}\right ]\,dx\,dy\\
I_2 &= \frac12\int_{\Omega\setminus B_{2r}}\int_{\Omega}\frac{|u(x)-u(y)|^{p-2}}{|x-y|^{n+sp}}\frac{u(x)-u(y)}{(u(x)+\delta)^{p-1}}\phi(x)^p\,dx\,dy,\\
I_3 &= \frac12\int_{B_{2r}}\int_{\Omega\setminus B_{2r}} \frac{|u(x)-u(y)|^{p-2}}{|x-y|^{n+sp}}\frac{(u(y)-u(x))}{(u(y)+\delta)^{p-1}}\phi(y)^p\,dx\,dy.\\
\end{align*}

Now arguing as in the proof of \cite[Lemma 1.3]{DiCastro-Kuusi-Palatucci} there exists a positive constant $C=C(p)$ such that 
\begin{align*}
&I_1 \leq-\frac{1}{C}\iint_{B_{2r}\times B_{2r}} \left|\log\left (\frac{u(x)+\delta}{u(y)+\delta}\right)\right|\frac{1}{|x-y|^{n+sp}}\, dxdy + Cr^{n-sp},\\
& I_2+I_3 \leq Cr^{n-sp}.
\end{align*}
These inequalities together with \eqref{usodebil} give us \eqref{deslogaritmica}.
\end{proof}

Using the above lemma we can enunciate the following version of the strong minimum principle for the operator $(-\Delta_{p,\Omega})^s$.
\begin{thm}[Strong minimum principle]\label{strongminimo}
Let $0<s<1<p<\ito$ be fixed and let $\Omega\subset \R^n$ be an open set. Let $A\subset \Omega$ be closed. Assume that $f\in L^{p'}(\Omega)$ is nonnegative and that $u\in W_A^{s,p}(\Omega)$ is a nonnegative weak solution of \eqref{mixed} in the sense of Definition \ref{de.mixto}. Then, either $u \equiv 0$ in $\Omega$ or $u>0$ almost everywhere in $\Omega$.
\end{thm}

\begin{proof}
See the proof in the Theorem A.1 in \cite{Brasco-Franzina}.
\end{proof}

\begin{rk}
Although it will not be needed in this paper, observe that the conclusions of Lemma \ref{logaritmic} and Theorem \ref{strongminimo} still hold for solutions of \eqref{eq.sp} in the sense of Definition \ref{de.sol}. The proof of these facts are completely analogous to that of Theorem \ref{strongminimo}.
\end{rk}

%%%%%%%%%%%%%%%%%%%%
%% HARD OBSTACLE PROBLEM
%%%%%%%%%%%%%%%%%%%%
\subsection{The hard obstacle problem} In this subsection we fix a measurable set $A\subset \Omega$ and prove the existence of an extremal for the constant $\lambda_s(A)$. Moreover, we show that this extremal is an eigenfunction of the regional $(s,p)-$laplacian in the sense of Definition \ref{de.mixto}.

Let us begin by showing that the constant $\lambda_s(A)$ is well defined and strictly positive.
\begin{prop}\label{poincare}
Let $\Omega\subset \R^n$ be a bounded, open set and let $A\subset \Omega$ be a measurable set with positive measure. Then there exists a constant $\theta>0$ depending on $n$, $s$, $p$, $\diam(\Omega)$ and $|A|$ such that $\lambda_s(A)\ge \theta$.
\end{prop}

\begin{proof}
The proof is rather simple. In fact, given $u\in W^{s,p}_A(\Omega)$ we have
\begin{align*}
[u]_{s,p}^p &= \iint_{\Omega\times \Omega} \frac{|u(x)-u(y)|^p}{|x-y|^{n+sp}}\, dxdy \ge \int_{\Omega} |u(x)|^p\left(\int_A \frac{1}{|x-y|^{n+sp}}\, dy\right)\, dx.
\end{align*}
So the proof will be completed is we can find a lower bound for the term
$$
 \int_A \frac{1}{|x-y|^{n+sp}}\, dy.
$$
So if we denote by $c(A;\Omega) = \sup\{|x-y|\colon x\in\Omega,\ y\in A\}$, we obtain
$$
 \int_A \frac{1}{|x-y|^{n+sp}}\, dy \ge c(A;\Omega)^{-(n+sp)}|A|.
$$
Therefore, we get $\lambda_s(A) \ge \frac12 c(A;\Omega)^{-(n+sp)}|A|$.

Observe that $c(A;\Omega)\le \diam(\Omega)$, so we can take $\theta = \frac12 \diam(\Omega)^{-(n+sp)} |A|$.
\end{proof}

Let us now see that there exists an extremal for the constant $\lambda_s(A)$. That is a function $u\in W^{s,p}_A(\Omega)$ such that
$$
\lambda_s(A) = \inf_{v\in W^{s,p}_A(\Omega)} \frac{\frac12 [v]_{s,p}^p}{\|v\|_p^p} = \frac{\frac12 [u]_{s,p}^p}{\|u\|_p^p}.
$$
This fact is a consequence of the compactness of the embedding $W^{s,p}(\Omega)\subset L^p(\Omega)$ (see the book of Maz'ya \cite{Mazya} or \cite[Theorem 7.1]{DiNezza-Palatucci-Valdinoci} for a direct proof). 
\begin{thm}\label{existe.extremal}
Let $\Omega\subset\R^n$ be a bounded domain with Lipschitz boundary. Then, given $A\subset \Omega$ measurable with positive measure, there exists $u\in W^{s,p}_A(\Omega)$ extremal for $\lambda_s(A)$. Moreover, the extremal can be taken to be normalized in $L^p(\Omega)$, i.e. $\|u\|_p = 1$.
\end{thm}

\begin{rk}
The Lipschitz regularity on $\partial\Omega$ is sufficient in order for the compactness of the embedding $W^{s,p}(\Omega)\subset L^p(\Omega)$ to hold. See  \cite[Theorem 7.1]{DiNezza-Palatucci-Valdinoci}. In fact, what is needed is that $\Omega$ be a {\em bounded extension domain}. That is the existence of a bounded extension operator $E\colon W^{s,p}(\Omega)\to W^{s,p}(\R^n)$. Lipschitz boundary implies that $\Omega$ is a bounded extension domain. See  \cite[Theorem 5.4]{DiNezza-Palatucci-Valdinoci}.
\end{rk}

\begin{proof}
The proof is immediate. We include some details for completeness.

Let $\{u_n\}_{n\in\N}\subset W^{s,p}_A(\Omega)$ be a normalized minimizing sequence for $\lambda_s(A)$. That is
$$
\|u_n\|_p = 1 \text{ for every } n\in \N \quad \text{and}\quad \lambda_s(A) = \lim_{n\to\ito} \frac12 [u_n]_{s,p}^p.
$$

Therefore, $\{u_n\}_{n\in\N}$ is a bounded sequence in $W^{s,p}(\Omega)$ and so since $W^{s,p}(\Omega)$ is a reflexive Banach space and from the compactness of the embedding into $L^p(\Omega)$ there exists a subsequence, that we still denote by $\{u_n\}_{n\in\N}$ and a function $u\in W^{s,p}(\Omega)$ such that
\begin{align}
\label{debil}
u_n\rightharpoonup u & \text{ weakly in } W^{s,p}(\Omega)\\
\label{fuerte}
u_n\to u & \text{ strongly in } L^p(\Omega).
\end{align}
It is easy to see that $u = 0$ a.e. in $A$ (for instance, taking a further subsequence, $u_n\to u$ a.e. in $\Omega$, or observe that $W^{s,p}_A(\Omega)$ is weakly closed since is strongly closed and convex), so $u\in W^{s,p}_A(\Omega)$.

By \eqref{fuerte}, it follows that $\|u\|_p=1$ and by \eqref{debil} and the weak lower semicontinuity of the Gagliardo seminorm,
$$
\lambda_s(A) \le \frac12 [u]_{s,p}^p \le \liminf_{n\to\ito} \frac12 [u_n]_{s,p}^p = \lambda_s(A).
$$
The proof is complete.
\end{proof}

Let us now see that an extremal of $\lambda_s(A)$ is an eigenfunction of the regional $(p,s)-$laplacian with eigenvalue $\lambda_s(A)$. That is, if $u\in W^{s,p}_A(\Omega)$ is an extremal for $\lambda_s(A)$, then
\begin{equation}\label{eq.EL.A}
\begin{cases}
(-\Delta_{p,\Omega})^s u = \lambda_s(A) |u|^{p-2}u & \text{in }\Omega\setminus A\\
u=0 & \text{in } A,
\end{cases}
\end{equation}
in the sense of Definition \ref{de.mixto}.

This is the content of the next result.
\begin{thm}
Let $0<s<1$ and $1<p<\ito$ be fixed. Let $\Omega\subset \R^n$ be a bounded open set and let $A\subset\Omega$ be measurable with positive measure. If $u\in W^{s,p}_A(\Omega)$ is an extremal for $\lambda_s(A)$, then $u$ is a solution to \eqref{eq.EL.A} in the sense of Definition \ref{de.mixto}.
\end{thm}

\begin{proof}
Let $u\in W_A^{s,p}(\Omega)$ be a normalized extremal for $\lambda_s(A)$, and let $v\in W_A^{s,p}(\Omega)$ be an arbitrary function. Define
\begin{equation}\label{def.j}
j(t)=\frac{1}{2}[u+tv]_{s,p}^p=\frac{1}{2}\iint_{\Omega\times \Omega}\frac{|(u+tv)(x)-(u+tv(y))|^p}{|x-y|^{n+sp}}\,dx\,dy.
\end{equation}
and 
\begin{equation}\label{def.k}
k(t)=\|u+tv\|_p^p=\int_\Omega |u+tv|^p\,dx.
\end{equation}
Then, is easy to see that
\begin{equation}\label{derj}
\begin{split}
j'(0) &=\frac{p}{2}\iint_{\Omega \times \Omega} \frac{|u(x)-u(y)|^{p-2}(u(x)-u(y))(v(x)-v(y))}{|x-y|^{n+sp}}\,dx\,dy\\
&= p \langle (-\Delta_{p,\Omega})^s u, v\rangle
\end{split}
\end{equation}
and
\begin{equation}\label{derk}
k'(0)=p \int_\Omega |u|^{p-2} uv\, dx.
\end{equation}
Moreover since $u\in W_A^{s,p}(\Omega)$ is a normalized extremal for $\lambda_s(A)$, we get 
\begin{equation}\label{normyext}
j(0) = \frac12 [u]_{s,p}^p = \lambda_s(A)\quad \text{and}\quad k(0)=\|u\|_p^p=1.
\end{equation}

So, if we define 
$$
G(t)=\frac{j(t)}{k(t)},
$$
we get that
$$
0=G'(0)=\frac{j'(0)k(0)-k'(0)j(0)}{[k(0)]^2}.
$$
Now using \eqref{derj}, \eqref{derk} and \eqref{normyext} in the above equality we obtain the desired result.
\end{proof}

%%%%%%%%%%%%%%%%%%%%
%% SOFT OBSTACLE PROBLEM
%%%%%%%%%%%%%%%%%%%%
\subsection{The soft obstacle problem} 
As in the previous subsection, we begin by showing that the constant $\lambda_s(\sigma, \phi)$ is positive. We prove this fact using the compactness of the embedding $W^{s,p}(\Omega)\subset L^p(\Omega)$ so we require $\partial\Omega$ to be Lipschitz continuous (cf. Theorem \ref{existe.extremal}).

\begin{thm}
Let $0<s<1$ and $1<p<\ito$ be fixed. Let $\Omega\subset\R^n$ be a bounded open set with Lipschitz boundary. Let $0\neq\phi\in L^\ito(\Omega)$ be nonnegative and let $\sigma>0$. Then, there exists a constant $\kappa>0$ such that $\lambda_s(\sigma,\phi)\ge \kappa$.
\end{thm}

\begin{proof}
The proof will follows if we show the existence of a constant $C>0$ such that
\begin{equation}\label{poincare.phi}
\int_\Omega |v|^p\, dx \le C\left(\frac12 \iint_{\Omega\times \Omega} \frac{|v(x)-v(y)|^p}{|x-y|^{n+sp}}\, dxdy + \sigma\int_\Omega |v|^p\phi\, dx.\right),
\end{equation}
for any $v\in W^{s,p}(\Omega)$.

Assume that \eqref{poincare.phi} is false. Then, there exists a sequence $\{v_k\}_{k\in\N}\subset W^{s,p}(\Omega)$ such that
\begin{align}
\label{vn.norm}
&\|v_k\|_p = 1,\\
\label{vn.gagliardo}
&[v_k]_{s,p}\to 0,\\
\label{vn.phi}
&\int_\Omega |v_k|^p\phi\, dx \to 0.
\end{align}

Now, arguing as in the proof of Theorem \ref{existe.extremal}, there exists a subsequence (that we still denote by $\{v_k\}_{k\in\N}$) and a function $v\in W^{s,p}(\Omega)$ such that
\begin{align*}
v_k\rightharpoonup v & \text{ weakly in } W^{s,p}(\Omega)\\
v_k\to v & \text{ strongly in } L^p(\Omega).
\end{align*}
Observe that from \eqref{vn.gagliardo} and the weak lower semicontinuity of the Gagliardo's seminorm, one conclude that
$$
0\le [v]_{s,p} \le \liminf_{n\to\infty} [v_k]_{s,p} = 0,
$$ 
so $v$ is constant. Also, from \eqref{vn.norm}, $\|v\|_p=1$ and so $v=|\Omega|^{-\frac{1}{p}}$.

But then, since $|v_k|^p\to |v|^p$ strongly in $L^1(\Omega)$ it follows that
$$
0 = \int_\Omega |v|^p\phi\, dx = \frac{1}{|\Omega|}\int_\Omega \phi\, dx,
$$
a contradiction.

The theorem is proved.
\end{proof}

With the same reasoning as in Theorem \ref{existe.extremal} it can be shown the existence of an extremal for $\lambda_s(\sigma,\phi)$. We state the theorem for future reference an leave the proof to the interested reader.
\begin{thm}\label{existe.extremal.phi}
Let $0<s<1$ and $1<p<\ito$ be fixed. Let $\Omega\subset\R^n$ be a bounded open set with Lipschitz boundary. Let $0\neq\phi\in L^\ito(\Omega)$ be nonnegative and let $\sigma>0$. Then, there exists an extremal $u\in W^{s,p}(\Omega)$ for $\lambda_s(\sigma,\phi)$.
\end{thm}

Finally, we show that an extremal for $\lambda_s(\sigma,\phi)$ is an eigenfunction of the regional $(s,p)-$laplacian in the sense of Definition \ref{de.sol}. This is the content of the next theorem.
\begin{thm}\label{teo.EL.phi}
Let $0<s<1<p<\ito$ be fixed, $\Omega\subset\R^n$ be a bounded open set, $0\neq\phi\in L^\ito(\Omega)$ be nonnegative and $\sigma>0$.  If $u\in W^{s,p}(\Omega)$ is an extremal for $\lambda_s(\sigma, \phi)$, then $u$ is a solution to
\begin{equation}\label{eq.EL.phi}
(-\Delta_{p,\Omega})^s u + \sigma\phi |u|^{p-2}u = \lambda_s(\sigma, \phi) |u|^{p-2}u \quad \text{in }\Omega,
\end{equation}
in the sense of Definition \ref{de.sol}.
\end{thm}
\begin{proof}
We suppose that $u\in W^{s,p}(\Omega)$ is a normalized extremal for $\lambda_s(\sigma,\phi)$. If $v\in W^{s,p}(\Omega)$, we define the function
$$
J(t)=\frac{j(t)+\sigma l(t)}{k(t)}
$$
where $j$ and $k$ are defined in \eqref{def.j} and \eqref{def.k} respectively, and 
$$
l(t)=\|u+tv\|_{p, \phi}^p=\int_\Omega|u+tv|^p \phi\,dx.
$$
Then
\begin{equation}\label{lprima}
l'(0)=\int_\Omega p |u|^{p-2}u v \phi\, dx.
\end{equation}
Now using that $k(0)=1$, $\lambda_s(\sigma,\phi)=j(0)+\sigma l(0)$, the expressions \eqref{derj}, \eqref{derk} and \eqref{lprima} we obtain
\begin{align*}
0&=J'(0)= \frac{(j'(0)+\sigma l'(0))k(0)-(j(0)+\sigma l(0))k'(0)}{k^2(0)}\\
&=p \langle (-\Delta_{p,\Omega})^s u, v\rangle +\sigma \int_\Omega p |u|^{p-2}uv\phi\,dx-\lambda_s(\sigma,\phi)p\int_\Omega |u|^{p-2}uv\,dx,
\end{align*}
as we wanted to prove.
\end{proof}

%%%%%%%%%%%%%%%%%%%%%%%%%
%%
%% OPTIMAL DESIGN PROBLEMS
%%
%%%%%%%%%%%%%%%%%%%%%%%%%
\section{Optimal design problems}

In this section we study the optimal design problems related to the hard and the soft obstacle problems. We devote one subsection to each of these problems and finally we analyze the connection between these two optimal design problems.

Since the compactness of the inclusion $W^{s,p}(\Omega)\subset L^p(\Omega)$ will be used throughout the section, we will always assume that $\Omega\subset \R^n$ is such that this compactness is guaranteed. For instance, that $\Omega$ is a bounded open set with Lipschitz boundary.

%%%%%%%%%%%%%%%%%%%%
%% HARD OBSTACLE PROBLEM
%%%%%%%%%%%%%%%%%%%%
\subsection{Optimization for the hard obstacle problem}
We begin this section showing that an extremal for $\lambda_s(A)$ has constant sign.
\begin{lemp}
Let $0<s<1<p<\ito$ be fixed. Let $\Omega\subset\R^n$ be a bounded open set and let $A\subset \Omega$ be measurable with positive measure. Then, if $u\in W^{1,p}_A(\Omega)$ is an extremal for $\lambda_s(A)$, it has constant sign, i.e. either $u\ge 0$ a.e. in $\Omega$ or $u\le 0$ a.e. in $\Omega$. 
\end{lemp}

\begin{proof}
The lemma is a consequence of the elementary inequality
$$
||a|-|b|| \begin{cases}
= |a-b| & \text{if } ab\ge 0\\
< |a-b| & \text{if } ab<0
\end{cases}
$$

In fact, let us denote $U_+ = \{u>0\}$ and $U_-= \{u<0\}$ and assume that $|U_\pm|>0$. Then
\begin{align*}
[u]_{s,p}^p &= \int_{\Omega}\int_{\Omega} \frac{|u(x) - u(y)|^p}{|x-y|^{n+sp}}\, dx\, dy \\
&= \int_{U_+}\int_{U_-}\frac{|u(x) - u(y)|^p}{|x-y|^{n+sp}}\, dx\, dy + \iint_{(\Omega\times\Omega)\setminus (U_+\times U_-)} \frac{|u(x) - u(y)|^p}{|x-y|^{n+sp}}\, dx\, dy \\
&\ge \int_{U_+}\int_{U_-}\frac{|u(x) - u(y)|^p}{|x-y|^{n+sp}}\, dx\, dy + \iint_{(\Omega\times\Omega)\setminus (U_+\times U_-)} \frac{||u(x)| - |u(y)||^p}{|x-y|^{n+sp}}\, dx\, dy \\
&> \int_{U_+}\int_{U_-}\frac{||u(x)| - |u(y)||^p}{|x-y|^{n+sp}}\, dx\, dy + \iint_{(\Omega\times\Omega)\setminus (U_+\times U_-)} \frac{||u(x)| - |u(y)||^p}{|x-y|^{n+sp}}\, dx\, dy\\
&= \int_{\Omega}\int_{\Omega} \frac{||u(x)| - |u(y)||^p}{|x-y|^{n+sp}}\, dx\, dy \\
&= [|u|]_{s,p}^p.
\end{align*}
Therefore, $u$ is not an extremal for $\lambda_s(A)$ which is a contradiction and the lemma is proved. 
\end{proof}

Before beginning with the proof of the existence of an optimal configuration we need a characterization of the constant $\Lambda_s(\alpha)$ given in \eqref{hard.opt}.
\begin{lemp}\label{caractLambda}
Let $\alpha$ be a number in $(0,1)$. Then
\begin{equation}\label{cteLambda}
\Lambda_s(\alpha)=\inf \left \{\frac{\frac12[u]_{s,p}^p}{\|u\|_{L^p(\Omega)}^p}\colon u\in W^{s,p}(\Omega), |\{u=0\}\cap\Omega|\geq \alpha |\Omega|\right \}.
\end{equation}
\end{lemp}
\begin{proof}
We define 
$$
\tilde{\Lambda}_s(\alpha):=\inf\left \{\frac{\frac12[u]_{s,p}^p}{\|u\|_{L^p(\Omega)}^p}\colon u\in W^{s,p}(\Omega), |\{u=0\}\cap\Omega|\geq \alpha |\Omega|\right \}.
$$
Let $A\subset \Omega$ be an arbitrary subset such that $|A|=\alpha|\Omega|$. By testing the quotient defining $\tilde{\Lambda}$ with the optimal function for $\lambda_s(A)$ we get that
$$
\tilde{\Lambda}_s(\alpha)\leq\lambda_s(A).
$$
Taking infimum in $A$ in the above inequality we obtain
\begin{equation}\label{des1} 
\tilde{\Lambda}_s(\alpha)\leq \Lambda_s(\alpha).
\end{equation}
On the other hand let $\{v_n\}_{n\in \N}$ be a normalized minimizing sequence for the constant $\tilde{\Lambda}_s(\alpha)$, i.e. $v_n\in W^{s,p}(\Omega)$, $\|v_n\|_{L^p(\Omega)}=1$,
$$
\tilde{\Lambda}_s(\alpha)=\lim_{n\to \ito} \frac12[v_n]_{s,p}^p,\quad |\{v_n=0\}\cap\Omega|\geq 0.
$$
Now for each $n\geq 1$, we take $A_n\subset \{v_n=0\}\cap \Omega$ such that $ |A_n|=\alpha |\Omega|$.

Thus 
$$
\Lambda_s(\alpha)\leq \lambda_s(A_n)\leq \frac12[v_n]_{s,p}^p,\quad \forall n\in \N.
$$
Taking the limit as $n\to\infty$ we obtain
\begin{equation}\label{des2}
\Lambda_s(\alpha)\leq\tilde{\Lambda}_s(\alpha).
\end{equation}
In this way \eqref{des1} and \eqref{des2} prove the lemma.
\end{proof}

Now we prove the existence of an optimal configuration for the problem \eqref{hard.opt}. 
\begin{thm}\label{existenciaoptima}
Let $\alpha$ be an arbitrary number in $(0,1)$. Then there exist:
\begin{enumerate}
\item A set $A\subset \Omega$, such that $|A|=\alpha|\Omega|$ and 
$$
\Lambda_s(\alpha)=\lambda_s(A).
$$
\item A function $u\in W^{s,p}(\Omega)$ with $|\{u=0\}\cap \Omega|\geq \alpha |\Omega|$, such that
$$
\Lambda_s(\alpha)=\frac{\frac12[u]_{s,p}^p}{\|u\|_{L^p(\Omega)}}.
$$
\end{enumerate}
\end{thm}
\begin{proof}
Clearly (1) follows immediately from (2). It suffices to take any set $A\subset \{u=0\}\cap\Omega$ such that $|A|=\alpha |\Omega|$.

Therefore, we only need to prove (2). Let $\{v_n\}_{n\in \N}$ be a normalized minimizing sequence for the constant $\Lambda_s(\alpha)$, i.e for each $n\in \N$.
$$
v_n\in W^{s,p}(\Omega),\quad  \|v_n\|_{L^p(\Omega)}=1,\quad  |\{ v_n=0\}\cap \Omega|\geq \alpha |\Omega|,\quad 
$$
and 
$$
\Lambda_s(\alpha)=\lim_{n\to \ito}\frac12[v_n]_{s,p}^p.
$$
In this way the sequence $\{v_n\}_{n\in \N}$ is a bounded sequence in the space $W^{s,p}(\Omega)$. Therefore by the reflexivity of the space $W^{s,p}(\Omega)$ and the compactness of the immersion $W^{s,p}(\Omega)\subset L^p(\Omega)$ (see \cite{DiNezza-Palatucci-Valdinoci}) there exists a function $u\in W^{s,p}(\Omega)$ and a subsequence, that we still denote by $\{v_n\}_{n\in\N}$, such that
\begin{align}
&v_n\rightharpoonup u \text{ weakly in } W^{s,p}(\Omega)\label{debilWsp}\\
&v_n \to u\text{ strongly in }L^p(\Omega)\label{fuerteLp}\\
&v_n\to u\text{ a.e. in }\Omega\label{ctp}.
\end{align}
By \eqref{fuerteLp} we can conclude that
$$
1=\lim_{n\to \ito}\|v_n\|_{L^p(\Omega)}=\|u\|_{L^p(\Omega)}.
$$ 
And by \eqref{ctp} and the upper semicontinuity of the measure of level sets, we obtain that
$$
|\{u=0\}\cap \Omega|\geq \lim_{n\to \ito}|\{v_n=0\}\cap\Omega|\geq \alpha |\Omega|.
$$
In this way the function $v$ is an admissible function for the constant $\Lambda_s(\alpha)$, then 
\begin{equation}\label{desSalpha}
\Lambda_s(\alpha)\leq \frac12[u]_{s,p}^p.
\end{equation}
Now using \eqref{debilWsp} and the lower semicontinuity of the seminorm $[\,\cdot\,]_{s,p}$ we get
$$
[u]_{s,p}\leq \liminf_{n\to\ito}[v_n]_{s,p}^p = 2\Lambda_s(\alpha).
$$
The above inequality and \eqref{desSalpha} tell us that the function $u$ satisfies (2).
\end{proof}

Theorem \ref{existenciaoptima} is not completely satisfactory. What one actually wants is that the optimal set $A_s$ coincides with the set where the extremal vanishes. Than is $|\{u=0\}\cap\Omega|=\alpha |\Omega|$ for any extremal $u\in W^{s,p}(\Omega)$. 

This fact will follows from the strong minimum principle for the regional $(s,p)-$la-placian proved in Theorem \ref{strongminimo}.
\begin{thm}\label{medida=}
If $u\in W^{s,p}(\Omega)$ is an extremal for $\Lambda_s(\alpha)$, then
$$
|\{u=0\}\cap\Omega|=\alpha|\Omega|.
$$
\end{thm}

\begin{proof}
We just have to prove that $|\{u=0\}\cap\Omega|\leq\alpha |\Omega|$. We assume for the rest of the proof that the extremal $u$ is normalized, so $\|u\|_{L^p(\Omega)}=1$ and that $u$ is nonnegative in $\Omega$.

Suppose by contradiction that $|\{u=0\}\cap\Omega|>\alpha |\Omega|$, then by the definition of the Lebesgue measure there exists a closed set $A\subset \{u=0\}\cap\Omega$ such that 
$$
|\{u=0\}\cap\Omega|>|A|>\alpha |\Omega|.
$$
Using the characterization of the constant $\Lambda_s(\alpha)$ we get
\begin{equation}
\Lambda_s(\alpha) = \frac12 [u]^p_{s,p} \le \lambda_s(A)\le \frac12 [u]_{s,p}^p = \Lambda_s(\alpha),
\end{equation}
where we have used the fact that $u$ is admissible in the characterization of $\lambda_s(A)$.

Then $u$ is an extremal for $\lambda_s(A)$ and therefore is a nonnegative solution for the problem 
\begin{equation}\label{formadebil}
\begin{cases}
(-\Delta_{p,\Omega}^s)u=\lambda |u|^{p-2}u,&\Omega\setminus A\\
u=0,&A,
\end{cases}
\end{equation}
in the sense of Definition \ref{de.mixto}.

Using Theorem \ref{strongminimo} we conclude that $u>0$ a.e. in $\Omega\setminus A$, but this is a contradiction since $|(\{u=0\}\cap\Omega)\setminus A|>0$. 

This completes the proof of the theorem.
\end{proof} 

%%%%%%%%%%%%%%%%%%%%
%% SOFT OBSTACLE PROBLEM
%%%%%%%%%%%%%%%%%%%%
\subsection{Optimization for the soft obstacle problem}

The problem that we consider now is to minimize the constant $\lambda_s(\sigma,\phi)$ on the class of bounded potential functions $\phi$, i.e. 
$$
\mathcal{B}=\{\phi\in L^{\ito}(\Omega)\colon 0\leq \phi\leq 1\}.
$$
This problem is trivial since in this case the infimum is given by $\phi=0$. So in order to have a nontrivial problem we consider, for $0<\alpha<1$, the functions $\phi\in\mathcal{A}$ with prescribed $L^1-$norm. So the problem to consider is
\begin{equation}\label{softobstacle}
\Lambda_s(\sigma,\alpha)=\inf\{\lambda_s(\sigma,\phi)\colon \phi\in\mathcal{B}_\alpha\},
\end{equation}
where 
$$
\mathcal{B}_\alpha=\left\{\phi\in \mathcal{B}\colon  \int_\Omega \phi\,dx=\alpha|\Omega|\right\}.
$$

Throughout this section, given $\phi\in \mathcal{B}_\alpha$ and $\sigma>0$, we use the notation
$$
I_{s, \phi,\sigma}(v) = \frac12 [v]_{s,p}^p + \sigma\int_\Omega |v|^p\phi\, dx,
$$
for every $v\in W^{s,p}(\Omega)$.

For this problem, a potential function $\phi\in \mathcal{B}_\alpha$ that realizes the infimum \eqref{softobstacle} is called an optimal potential and if $u$ is an eigenfunction for  $\lambda_s(\sigma, \phi)$ with $\phi$ an optimal potential, the pair $(u,\phi)$ is called an optimal pair for this problem.

We show the existence of an optimal potential for the problem \eqref{softobstacle} and, moreover, we show that every optimal potential is a characteristic function.

The following theorem gives us the existence of an optimal configuration.
\begin{thm}\label{softoptimal}
Let $0<s<1<p<\ito$, $\sigma>0$ and $0<\alpha<1$ be fixed. Let $\Omega\subset \R^n$ be a bounded open set with Lipschitz boundary. Then there exists an optimal pair $(u,\phi)\in W^{s,p}(\Omega)\times \mathcal{B}_\alpha$. Moreover one can take $\phi=\chi_D$, with
$$
\{u<t\}\subset D\subset \{u\leq t\},
$$
for some $t>0$.
\end{thm}
\begin{proof}
Let $\{\phi_k\}_{k\in\N}$ be a minimizing sequence in $\mathcal{B}_\alpha$ for $\Lambda_s(\sigma,\alpha)$ i.e.
$$
\Lambda_s(\sigma,\alpha)=\lim_{k\to\ito}\lambda_s(\sigma,\phi_k).
$$

The sequence $\{\phi_k\}_{k\in\N}$ is a bounded sequence in $L^\ito(\Omega)$, then using the fact that $L^1(\Omega)$ is a separable Banach space and $L^\ito(\Omega)$ is its dual space, there exists $\phi\in L^\ito (\Omega)$ and a subsequence (that still we call $\{\phi_k\}_{k\in\N}$) such that
\begin{align}\label{debilestrellaphik}
\phi_k \stackrel{\ast}{\rightharpoonup} \phi \text{ in } L^\ito(\Omega).
\end{align}
Observe that \eqref{debilestrellaphik} implies that $\phi\in \mathcal{B}_\alpha$.

Now let $\{u_k\}_{k\in\N}\in W^{s,p}(\Omega)$ be the corresponding sequence of normalized eigenfunctions for the constants $\lambda_s(\sigma,\phi_k)$. That is $\lambda_s(\sigma,\phi_k)=I_{s,\phi_k,\sigma}(u_k)$. 

Since $[u_k]_{s,p}^p\le 2 I_{s,\phi_k,\sigma}(u_k)$ and $\|u_k\|_p=1$, then the extremals $\{u_k\}_{k\in\N}$ are uniformly bounded in $W^{s,p}(\Omega)$. By the reflexivity of the space $W^{s,p}(\Omega)$ (see \cite[p. 205]{Adams}) and the compactness of the embedding $W^{s,p}(\Omega)\subset L^p(\Omega)$ there exists a subsequence (that we still call $\{u_k\}_{k\in\N}$) and a function $u\in W^{s,p}(\Omega)$ such that
\begin{align}
& u_k \rightharpoonup u \text{ weakly in } W^{s,p}(\Omega)\\
& u_k\to u \text{ strongly in }L^p(\Omega)\label{Lpsoft} \\
& u_k\to u \text{ a.e. in } \Omega. \label{ctpsoft}
\end{align}
Using \eqref{Lpsoft} we get
$$
1=\lim_{k\to\ito}\|u_k\|_{L^p(\Omega)}=\|u\|_{L^p(\Omega)},
$$
thus $u$ is an admissible function in the definition of the constant $\lambda_s(\sigma,\phi)$, in consequence
\begin{equation}\label{admisiblesoft}
\lambda_s(\sigma,\phi)\leq I_{s,\phi,\sigma}(u).
\end{equation}
Now using \eqref{Lpsoft}, it follows that $|u_k|^p\to |u|^p$ strongly in $L^1(\Omega)$ (this follows directly from H\"older's inequality together with the inequality $|a^p-b^p|\le p |a^{p-1}-b^{p-1}| |a-b|$ for any $a,b\ge 0$ and $p\ge 1$) and therefore, by \eqref{debilestrellaphik} we obtain
$$
\int_\Omega |u_k|^p\phi_k \, dx \to \int_\Omega |u|^p\phi \, dx \text{ as } k\to\infty.
$$
Finally, by the lower semicontinuity of the Gagliardo seminorm $[\,\cdot\,]_{s,p}$ we get
\begin{equation}\label{liminfI}
I_{s,\phi,\sigma}(u)\leq \liminf_{k\to\ito}I_{s,\phi_k,\sigma}(u_k)
\end{equation}

Using first \eqref{liminfI} and then \eqref{admisiblesoft}, we obtain
$$
I_{s,\phi,\sigma}(u)\leq \liminf_{k\to\ito}I_{s,\phi_k,\sigma}(u_k)=\Lambda_s(\sigma,\alpha)\leq \lambda_s(\sigma,\phi)\leq I_{s,\phi,\sigma}(u).
$$
Then $\Lambda_s(\sigma,\alpha)=\lambda_s(\sigma,\phi)=I_{s,\phi,\sigma}(u)$. This shows that $(u,\phi)\in W^{s,p}(\Omega)\times \mathcal{B}_\alpha$ is an optimal pair.

Finally using  the {\em Bathtub principle} (see \cite[Theorem 1.14, p. 28]{Lieb-Loss}) we conclude that the problem
$$
\inf \left \{\int_\Omega |u|^p\phi\,dx\colon  \phi\in \mathcal{B}_\alpha\right \}
$$
has a solution of the form $\phi=\chi_D$ with $\{u<t\}\subset D\subset \{u\leq t\}$ for some $t>0$.
\end{proof}

\begin{rk}
Again by a direct application of the Bathtub principle, the uniqueness of the optimal potential is equivalent to the fact that $|\{u=t\}|=0$.
\end{rk}

%%%%%%%%%%%%%%%%%%%%
%% CONNECTION
%%%%%%%%%%%%%%%%%%%%
\subsection{Connection between soft and hard problems}
In this subsection we make rigorous the fact that the soft obstacle problem is a penalized version of the hard obstacle problem and that the soft obstacle problem converges to the hard obstacle one when the penalization term goes to $+\ito$.

In fact, what we prove here is that $\Lambda_s(\sigma,\alpha)\to \Lambda_s(\alpha)$ when $\sigma \to \ito$, establishing a connection between the optimal problems \eqref{hard.opt} and \eqref{soft.opt}.

\begin{thm}\label{kappaainfinito}
Let $0<s<1<p<\ito$ and let $0<\alpha<1$ be fixed. Let $\Omega\subset \R^n$ be a bounded open set with Lipschitz boundary and let $\Lambda_s(\alpha)$ and $\Lambda_s(\sigma, \alpha)$ be the constants defined in \eqref{hard.opt} and \eqref{soft.opt} respectively. Then
$$
\lim_{\sigma\to\ito} \Lambda_s(\sigma,\alpha) = \Lambda_s(\alpha).
$$

Moreover, if we denote by $(u_\sigma,\chi_{D_\sigma})\in  W^{s,p}(\Omega)\times \mathcal{B}_\alpha$ an optimal pair for $\Lambda_s(\sigma, \alpha)$, then the family $\{(u_\sigma, \chi_{D_\sigma})\}_{\sigma>0}$ is precompact in the strong topology of $W^{s,p}(\Omega)$ in the first variable and the weak* topology of $L^\ito(\Omega)$ in the second variable.

Finally, any accumulation point of the family has the form $(u,\chi_D)$ where $D\in \mathcal{A}_\alpha$, $\{u=0\}\cap\Omega=D$ and $u$ is an extremal for $\Lambda_s(\alpha)$.
\end{thm}

\begin{proof} Given $0<\alpha<1$, let $(u,A)$ be an optimal configuration for the constant $\Lambda_s(\alpha)$. Than is $u\in W^{s,p}_A(\Omega)$, $\|u\|_p=1$, $|A|=\alpha|\Omega|$ and $\Lambda_s(\alpha)= \frac12 [u]_{s,p}^p$.

Hence, we have that, for every $\sigma>0$,
$$
\Lambda_s(\sigma, \alpha) \le \lambda_s(\sigma, \chi_A)\le I_{s,\chi_A, \sigma}(u) = \frac12[u]_{s,p}^p = \Lambda_s(\alpha).
$$
Since trivialy $\Lambda_s(\sigma, \alpha)$ is nondecreasing with $\sigma$, it follows that there exists $\Lambda_*$ such that
\begin{equation}\label{Lambda*}
\Lambda_s(\sigma, \alpha)\uparrow \Lambda_*\le \Lambda_s(\alpha).
\end{equation}

Now let $\{\sigma_j\}_{j\in\N}\subset\R$ be a sequence such that $\sigma_j\uparrow \infty$ and take $(u_j,\phi_j)$ to be an optimal pair for $\Lambda_s(\sigma_j,\alpha)$. Then, for every $j\in\N$,
$$
\frac12 [u_j]_{s,p}\le I_{s,\phi_j,\sigma_j}(u_j)=\Lambda_s(\sigma_j,\alpha)\leq \Lambda_s(\alpha).
$$

Since $\|u_j\|_p=1$ for every $j\in\N$, using the reflexivity of the space $W^{s,p}(\Omega)$, the compactness of the embedding $W^{s,p}(\Omega)\subset L^p(\Omega)$ and the fact that $L^\ito(\Omega)$ is the dual of $L^1(\Omega)$, there exist a pair $(u,\phi)\in W^{s,p}(\Omega)\times \mathcal{B}$ and a subsequence (that we still call $\{(u_j,\phi_j)\}_{j\in\N}$) such that
\begin{align}
&\label{debilujsigma} u_j\rightharpoonup u \text{ weakly in } W^{s,p}(\Omega)\\
&\label{fuerteujLpsigma}u_j\to u \text{ strongly in }L^p(\Omega)\\
&\label{debilphijsigma} \phi_j \stackrel{\ast}{\rightharpoonup} \phi \text { $\ast-$weakly in }L^\ito(\Omega).
\end{align}

Using \eqref{fuerteujLpsigma} we get $1=\lim_{j\to \ito}\|u_j\|_{L^p(\Omega)}=\|u\|_{L^p(\Omega)}$, and by \eqref{debilphijsigma} we obtain that $\phi \in \mathcal{B}_\alpha$. Also using \eqref{fuerteujLpsigma} and \eqref{debilphijsigma} 
\begin{equation}\label{convobstaculo}
\int_\Omega |u_j|^p \phi_j\, dx \to \int_\Omega |u|^p \phi \, dx, \quad \text{as } j\to \ito.
\end{equation}
Taking into account \eqref{Lambda*}, we get
$$
0\leq \sigma_j\int_\Omega \phi_j |u_j|^p\,dx\leq I_{s,\phi_j,\sigma_j}(u_j)=\Lambda_s(\sigma_j,\alpha)\leq \Lambda_s(\alpha),
$$
from where we obtain
$$
0\leq \int_\Omega |u_j|^p \phi_j \, dx \leq \frac{\Lambda_s(\alpha)}{\sigma_j}
$$
and so
$$
\int_\Omega |u_j|^p \phi_j \, dx \to 0, \quad \text{as } j\to \ito.
$$
This and \eqref{convobstaculo} produce
\begin{equation}\label{uphiiguala0}
\int_\Omega |u|^p \phi \, dx = 0,
\end{equation}
from where we conclude that
\begin{equation}\label{phiuigual0}
|u|^p \phi = 0\quad \text{a.e. in } \Omega
\end{equation}

Now, \eqref{phiuigual0} implies that $u$ is admisible in the characterization of $\Lambda_s(\alpha)$. In fact, if we denote by $D=\{\phi>0\}\cap \Omega$, then
\begin{equation}\label{estimate.D}
|D| \ge \int_D \phi\, dx = \int_\Omega \phi\, dx = \alpha|\Omega|,
\end{equation}
and by \eqref{phiuigual0}, $u=0$ a.e. in $D$ as so $|\{u=0\}\cap\Omega|\ge |D|\ge \alpha |\Omega|$ as we wanted to show.

By \eqref{Lambda*}, \eqref{uphiiguala0} and the semicontinuity of the Gagliardo seminorm $[\,\cdot\,]_{s,p}$, 
\begin{align*}
\Lambda_s(\alpha)\ge \Lambda_* = &\lim_{j\to\ito}\Lambda_s(\sigma_j,\alpha)\\
=&\lim_{j\to\ito}I_{s,\phi,\sigma}(u_j)\\
\geq &\liminf_{j\to\ito}\frac12 [u_j]_{s,p}^p\\
\geq & \frac12 [u]_{s,p}^p\ge \Lambda_s(\alpha).
\end{align*}

Observe that the above computation shows that $u$ is an extremal for $\Lambda_s(\alpha)$. Moreover, since $\{u=0\}\cap\Omega\supset D$ and $\alpha|\Omega|\le |D|\le |\{u=0\}\cap\Omega|= \alpha |\Omega|$, where we have used Theorem \ref{medida=} in the last equality, we deduced that $D=\{u=0\}\cap\Omega$. Moreover, from \eqref{estimate.D} we easily deduce that $\phi = \chi_D$.

It remains to see that in fact $u_j\to u$ strongly in $W^{s,p}(\Omega)$. But from the above computations is easy to check that 
$$
\lim_{j\to\infty} [u_j]_{s,p} = [u]_{s,p},
$$
therefore, by \eqref{fuerteujLpsigma},
$$
\lim_{j\to\ito}\|u_j\|_{W^{s,p}(\Omega)} = \|u\|_{W^{s,p}(\Omega)}.
$$
Since $W^{s,p}(\Omega)$ is a uniformly convex Banach space, it follows that $u_j\to u$ strongly in $W^{s,p}(\Omega)$.

This finishes the proof of the theorem.
\end{proof}

%%%%%%%%%%%%%%%%%%%%%%%%%
%%
%% S to 1
%%
%%%%%%%%%%%%%%%%%%%%%%%%%
\section{Asymptotic behavior for $\Lambda_s(\alpha)$ with $s\uparrow 1$} In this section we consider $\Omega\subset \R^n$ a bounded open set with Lipschitz boundary, $A\subset \Omega$ measurable with positive measure and define the functions $I_s, I\colon L^p(\Omega)\to \R$ by
\begin{equation}\label{Is}
I_s(u)=[u]_{s,p}^p=\begin{cases}
\displaystyle \iint_{\Omega\times \Omega} \frac{|u(x)-u(y)|^p}{|x-y|^{n+sp}}\,dx\,dy & \text{if } u\in W^{s,p}(\Omega)\cap E_\alpha\\
+\infty & \text{otherwise}
\end{cases}
\end{equation}
\begin{equation}\label{lIp}
I(u)= \begin{cases}
\displaystyle \int_\Omega|\nabla u|^p\,dx& \text{if } u\in W^{1,p}(\Omega)\cap E_\alpha\\
+\infty & \text{otherwise},
\end{cases}
\end{equation}
where  $E_\alpha = \{v\in L^p(\Omega)\colon |\{v=0\}\cap \Omega|\ge \alpha|\Omega|, \|v\|_p=1\}$. 

In this way is satisfied that 
\begin{equation}\label{expresionlambda.s}
\Lambda_{s}(\alpha)=\inf_{v\in L^p(\Omega)} \frac12 I_s(v)
\end{equation}
and we define the constant
\begin{equation}\label{expresionlambda}
\Lambda(\alpha)=\inf_{v\in L^p(\Omega)} \frac12 I(v).
\end{equation}
The idea is to analyze the asymptotic behavior of the constants $\Lambda_s$ by showing that, properly rescaled, they converge to the constant $\Lambda$ when $s\uparrow 1$. By \eqref{expresionlambda.s} and \eqref{expresionlambda} this is equivalent to analyze the asymptotic behavior of the infimums. To do it the right tool that we need to use is the concept of $\Gamma-$convergence (see \cite{Braides, DalMaso}). 
\begin{de}($\Gamma-$convergence)
Let $(X,d)$ be a metric space, and for every $j\in\N$, let $F_j,F\colon X\to \bar \R$. We say that the functions $F_j$ $\Gamma-$converges to the function $F$ if for every $x\in X$ the following conditions are valid.
\begin{enumerate}
\item($\liminf$ inequality) For every sequence $\{x_j\}_{j\in\N}\subset X$ such that $x_j\to x$, it holds
$$
F(x)\leq \liminf_{j\to \ito} F_j(x_j).
$$
\item ($\limsup$ inequality) There exists a sequence $\{y_j\}_{j\in\N}\subset X$ with $y_j\to x$ such that 
$$
\limsup_{j\to \ito}F_j(y_j)\leq F(x).
$$
\end{enumerate}
The function $F$ is called the $\Gamma-$limit of the sequence $\{F_j\}_{j\in\N}$, and this is denoted by 
$$
\glim_{j\to \ito}{F_j}=F.
$$
\end{de}
The following theorem, whose proof is elementary (see \cite{Braides}), will be most helpful in the sequel. We remark that this is not the most general result that can be obtained, nevertheless it will suffices for our purposes. For a comprehensive analysis of $\Gamma-$convergence, we refer to the book of G. Dal Maso \cite{DalMaso}.
\begin{thm}[Convergence of minima]\label{convergenciademinimos}
Let $(X,d)$ be a complete metric space, let $F_j, F\colon X\to \R$ be functions such that $\glim_{j\to \ito}{F_j}=F$ and suppose that there exist  $\{x_j\}_{j\in \N}\subset X$ such that
$$
\inf_X F_j=F_j(x_j) + o(1).
$$
Assume that the sequence $\{x_j\}_{j\in\N}$ is precompact in $X$. Then 
$$
\inf_XF_j\to \min_X F\quad  \text{if } j\to \ito.
$$
Moreover, if $x_0\in X$ is an accumulation point of the sequence $\{x_j\}_{j\in\N}$, then $x_0$ is a minimum for $F$.
\end{thm}

The idea is now to use Theorem \ref{convergenciademinimos} to obtain our convergence result. First we show the $\Gamma-$convergence of the functionals which is a simple derivation from \cite[Theorem 8]{Ponce}. 
\begin{thm}\label{Ponce}
Let $I_s, I\colon L^p(\Omega)\to \bar\R$ be the functions defined in \eqref{Is} and \eqref{lIp} respectively. Then
$$
\glim_{s\to 1}\, (1-s)I_s = K(n,p) I,
$$
where $K(n,p)=\pint_{S^{n-1}}|e_1\cdot \sigma|^p\,d \mathcal{H}^{n-1}(\sigma) = \frac{\Gamma(\frac{n}{2})\Gamma(\frac{p+1}{2})}{\sqrt{\pi}\Gamma(\frac{n+p}{2})}$.
\end{thm}
Before starting with the proof let us make the following observation.  Let $J_s,J\colon L^p(\Omega)\to \bar{\R}$ be defined as 
\begin{equation}\label{Isbis}
J_s(u)=\begin{cases}
\displaystyle \iint_{\Omega\times \Omega} \frac{|u(x)-u(y)|^p}{|x-y|^{n+sp}}\,dx\,dy & \text{if } u\in W^{s,p}(\Omega)\\
+\infty & \text{otherwise}
\end{cases}
\end{equation}
\text{and}
\begin{equation}\label{lIpbis}
J(u)= \begin{cases}
\displaystyle \int_\Omega|\nabla u|^p\,dx& \text{if } u\in W^{1,p}(\Omega)\\
+\infty & \text{otherwise}
\end{cases}
\end{equation}
respectively. Then, in \cite[Theorem 8]{Ponce}, it is proved that $\glim_{s\to 1} (1-s)J_s = K(n,p)J$.

Observe that this result does not imply directly Theorem \ref{Ponce}, since our functions $I_s$ and $I$ are restrictions of $J_s$ and $J$ respectively and restrictions of $\Gamma-$converging functions do not necessarily $\Gamma-$converge. See \cite[Proposition 6.14]{DalMaso}.

However, in our case we can still recover Theorem \ref{Ponce} from \cite[Theorem 8]{Ponce} as a consequence of this general result.

\begin{lemp}\label{restriccion}
Let $F_j,F\colon X\to\bar{\R}$ be functions such that $\glim_{j\to \ito}{F_j}=F$. Let $Y\subset X$  be closed and define the restricted functions $\tilde F_j, \tilde F\colon X\to\bar\R$ by
$$
\tilde{F}_j(x)=
\begin{cases}
F_j(x) \text{ if } x\in Y,\\
+ \ito  \text{ otherwise,} 
\end{cases} \quad\text{and}\quad 
\tilde{F}(x)=
\begin{cases}
F(x) \text{ if } x\in Y,\\
+ \ito  \text{ otherwise.} 
\end{cases}
$$
If, moreover $F_j(y)\to F(y)$ for every $y\in Y$, then 
$$
\glim_{j\to \ito}\tilde{F}_j=\tilde{F}.
$$
\end{lemp}
\begin{proof}
We first prove the $\liminf$ inequality. Indeed let $x\in X$ and $x_j\to x$ in $X$, we must prove that 
\begin{equation}\label{delliminf}
\tilde{F}(x)\leq \liminf_{j\to \ito } \tilde{F}_j(x_j).
\end{equation}
We can assume that $\liminf_{j\to\ito} \tilde{F}_j(x_j)<+\ito$, otherwise there is nothing to prove.

Then, we can assume that $x_j\in Y$ for every $j\in\N$, and so
$$
\tilde{F_j}(x_j) = F_j(x_j).
$$
Now using the $\Gamma-$convergence of the functions $F_j$ we get
$$
F(x)\le \liminf_{j\to \ito}F_j(x_j) = \liminf_{j\to \ito}\tilde{F}_j(x_j)
$$
Since $Y$ is closed, we conclude that $x\in Y$ and so $F(x)=\tilde{F}(x)$. In this way \eqref{delliminf} is proved.

Finally, we prove the $\limsup$ inequality i.e. for each $x\in X$ we must prove that exist a sequence $\{x_j\}_j\in X$ such that $x_j \to x$ in $X$ and 
\begin{equation}\label{delimisup}
\limsup_{j\to \ito}\tilde{F}_j(x_j)\leq \tilde{F}(x).
\end{equation}
We take the constant sequence $x_j=x$. If $x\notin Y$ then $\tilde{F}_j(x_j)= \tilde{F}(x)=\ito$ and the above inequality is trivial. We suppose that $x\in Y$, in this case $\tilde{F}_j(x_j)=F_j(x)$ and $\tilde{F}(x)=F(x)$, then by the pointwise convergence hypothesis we get that 
$$
\lim_{j\to \ito}F_j(x)=F(x),
$$     
in particular \eqref{delimisup} is valid.
\end{proof}

With this lemma, the proof of Theorem \ref{Ponce} is trivial.
\begin{proof}[Proof of Theorem \ref{Ponce}]
Let $\{s_j\}_j$ be an arbitrary sequence in $(0,1)$ with $s_j\to 1$ when $j\to \ito$. We will apply Lemma \ref{restriccion} with
$$
F_j=(1-s_j)J_{s_j},\quad F=K(n,p) J,\quad X=L^p(\Omega) \quad \text{and}\quad Y=L^p(\Omega)\cap E_\alpha.
$$
Then,
$$
\tilde{F}_j=(1-s_j)I_{s_j},\quad \tilde{F}=K(n,p) I.
$$
In order to apply Lemma \ref{restriccion} we need prove that
\begin{equation}\label{puntual1}
F_j(v)\to F(v),\quad \forall v\in L^p(\Omega)\cap E_\alpha
\end{equation}
and that the space $Y=L^p(\Omega)\cap E_\alpha$ is closed in $L^p(\Omega)$.

Now \eqref{puntual1} is proved in \cite[Corollary 2]{Bourgain-Brezis-Mironescu}.  On the other hand let $\{u_j\}_{j\in \N}\in Y$ be such that $u_j\to u$ in $L^p(\Omega)$ as $j\to\ito$. This implies that
$$
|\{u=0\}|\ge \limsup_{j\to\ito} |\{u_j=0\}| \ge \alpha |\Omega|,
$$
by the semicontinuity of the measure.

The proof is now complete.
\end{proof}

In this way the convergence of the eigenvalues is a direct consequence of Theorems \ref{convergenciademinimos} and \ref{Ponce}. 
\begin{thm}\label{mainasintotico}
Let $0<\alpha<1$ be fixed and let the constants $\Lambda_s(\alpha)$ and $\Lambda(\alpha)$ be defined by \eqref{expresionlambda.s} and \eqref{expresionlambda} respectively. Then
$$
(1-s)\Lambda_s(\alpha)\to K(n,p)\Lambda(\alpha)\quad \text{as}\quad s\uparrow 1.
$$
Moreover, if $u_s\in W^{s,p}(\Omega)\cap E_\alpha$ is an extremal for $\Lambda_s(\alpha)$, then for every sequence $s_j\to 1$, $\{u_{s_j}\}_{j\in\N}$ is precompact in $L^p(\Omega)$ and every accumulation point of the sequence is an extremal for $\Lambda(\alpha)$.
\end{thm}

\begin{proof}
Theorem \ref{Ponce}, gives the $\Gamma-$convergece of the functionals $I_s$ to $I$. In order to apply Theorem \ref{convergenciademinimos} it remains to check that a sequence of minimizers for $I_s$ is precompact in $L^p(\Omega)$. But this is proved in \cite[Corollary 7]{Bourgain-Brezis-Mironescu}.
\end{proof}

\subsection{Convergence of optimal sets} An interesting consequence of the above theorem is that the sequence of extremals $u_s$ associated to the constants $\Lambda_s(\alpha)$ verify that $(1-s)[u_s]_{s,p}^p$ is uniformly bounded. Then by the compactness result \cite[Corollary 7]{Bourgain-Brezis-Mironescu}, there exists a function $u\in W^{1,p}(\Omega)$ and a sequence $u_{s_j}$ such that $u_{s_j} \to u$ strongly in $L^p(\Omega)$. This fact allows us conclude the convergence of the optimal sets $A_{s_j} = \{u_{s_j}=0\}$ to an optimal set $A$ associated to the constant $\Lambda(\alpha)$.

We first need the follow result that was obtained in \cite{Bonder-Groisman-Rossi}.
\begin{lemp}[\cite{Bonder-Groisman-Rossi}, Lemma 3.1]\label{convsimetrica}
Let $(X,\Sigma,\nu)$ be a measure space of finite measure and let $\{f_n\}_{n\in \N},f$ be $\nu-$measureble nonnegative functions such that $f_n\to f$ a.e.

Let $\mu$ a nonnegative measure absolutely continuous with respect to $\nu$. Then if $\lim_{n\to \ito}\mu(\{f_n=0\})=\mu(\{f=0\})$, it follows that
$$
\lim_{n\to \ito}\mu(\{f_n=0\}\triangle \{f=0\})=0.
$$  
\end{lemp}

Another property that we need is the fact that any extremal $u\in W^{1,p}(\Omega)\cap E_\alpha$ for $\Lambda(\alpha)$ has the property that $|\{u=0\}|=\alpha|\Omega|$. This follows in a completely analogous way as in \cite{Bonder-Rossi-Wolanski1}. We state the result for future reference.
\begin{lemp}\label{extremal1}
Let $u\in W^{1,p}(\Omega)\cap E_\alpha$ be an extremal for $\Lambda(\alpha)$. Then 
$$
|\{u=0\}|=\alpha|\Omega|.
$$
\end{lemp}

\begin{proof}
See the proof of Theorem 1.2 in \cite{Bonder-Rossi-Wolanski1}.
\end{proof}

Now using the above lemmas we can prove the convergence of optimal sets.
\begin{thm}\label{conventanas}
Let $A_s\subset \Omega$ be an optimal set for $\Lambda_s(\alpha)$. Then, there exists a sequence $s_j\to 1$ and a set $A\subset \Omega$ such that
$$
\chi_{A_{s_j}}\to \chi_{A}\text{ strongly in } L^1(\Omega).
$$
Moreover, the set $A$ is optimal for $\Lambda(\alpha)$ in the sense that there exists an extremal $u\in W^{1,p}(\Omega)\cap E_\alpha$ such that $\{u=0\}=A$.
\end{thm}
\begin{proof}
Let $u_s\in W^{s,p}(\Omega)\cap E_\alpha$ be an extremal for $\Lambda_s(\alpha)$. Then, by Theorem \ref{medida=} we have
$$
|\{u_s=0\}\cap \Omega|=\alpha|\Omega|.
$$
and by Theorem \ref{mainasintotico}, there exists a sequence $s_j\to 1$ and a function $u\in W^{1,p}(\Omega)\cap E_\alpha$ such that $u_{s_j}\to u$ strongly in $L^p(\Omega)$ and $u$ is an extremal for $\Lambda(\alpha)$.

By Lemma \ref{convsimetrica} and Lemma \ref{extremal1} we obtain
\begin{equation}\label{triangle}
\lim_{j\to\ito} |\{u_{s_j}=0\}\triangle \{u=0\}| = 0.
\end{equation}
So if we called $A_s := \{u_s=0\}$ and $A=\{u=0\}$ from \eqref{triangle} we obtain
$$
\chi_{A_{s_j}}\to \chi_A\quad \text{strongly in } L^1(\Omega).
$$
The proof is complete.
\end{proof}

\section*{Acknowledgements}

This paper was partially supported by Universidad de Buenos Aires under grant UBACyT 20020130100283BA, CONICET under grant PIP 11220150100032CO and by ANPCyT under grant PICT 2012-0153. J. Fern\'andez Bonder is a member of CONICET.

We want to thank L. Del Pezzo for his help with Section 2. In particular, the proof of Lemma \ref{plap.bien} was communicated to us by him.

We also want to thank the anonymous referee for the careful reading of the manuscript and several suggestions that help us to improve the presentation of the paper.

%%%%%%%%%%%%%%%%%%%
%%
%% BIBLIOGRAFIA
%%
%%%%%%%%%%%%%%%%%%%

\bibliographystyle{plain}
\bibliography{bib}

\end{document}